\documentclass{article}
\usepackage{latexsym,amsfonts,amsthm,amsmath,mathrsfs}
\newtheorem{theo}{Theorem}[section]
\newtheorem{lem}{Lemma}[section]

\newtheorem{prop}{Proposition}[section]
\newtheorem{cor}{Corollary}[section]

\theoremstyle{definition}
\newtheorem{defi}{Definition}
\newtheorem{rem}{Remark}[section]

\newcommand{\R}{\mathbb{R}}
\newcommand{\RN}{\mathbb{R}^{N+1}_+}

\newcommand{\N}{\mathbb{N}}
\newcommand{\la}{\lambda}
\newcommand{\ve}{\varepsilon}

\title{The pseudorelativistic Hartree equation with a general nonlinearity: existence, non existence and variational identities}
\author{{\sc Dimitri Mugnai}\thanks{Research supported by the Miur PRIN {\sl Variational methods and nonlinear PDEs}}\\Dipartimento di Matematica e
Informatica\\Universit\`a di Perugia\\Via Vanvitelli 1, 06123
Perugia - Italy\\ tel. +39 075 5855043, fax. +39 075 5855024,\\
e-mail: mugnai@dmi.unipg.it}

\date{}
\begin{document}

\maketitle
\begin{abstract}
We prove several existence and non existence results of solitary
waves for a class of nonlinear pseudo--relativistic Hartree
equations with general nonlinearities. We use variational methods
and some new variational identities involving the half Laplacian.
\end{abstract}


2000AMS Subject Classification: 35Q55, 35A15, 35J20, 35Q40, 35Q85

\section{Introduction, motivations and main result}

In \cite{FL}, Fr\"ohlich e Lenzmann studied the Schr\"odinger
equation with a Hartree nonlinearity
\begin{equation}\label{1FL}
i\psi_t=\sqrt{-\Delta+m^2}\psi-\left(\frac{1}{|x|}\ast
|\psi|^2\right)\psi \quad \mbox{ in }\R^3
\end{equation}
as a description of pseudorelativistic boson stars (see \cite{2} for
the rigorous derivation of the model). Here $\psi(t,x)$ is a
complex--valued wave field (a one-particle wave function), the
symbol $\ast$ stands for the usual convolution in $\R^3$, and
$1/|x|$ is the Newtonian gravitational potential (after setting all
physical constants equal to 1). The operator $\sqrt{-\Delta+m^2}$,
which coincides with the so--called {\em half Laplacian}
$(-\Delta)^{1/2}$ when $m=0$, describes the kinetic and rest energy
of a relativistic particle of mass $m\geq0$\footnote{The kinetic
energy of a fermion with mass $m
> 0$ is described by the pseudo--differential operator
$\sqrt{-\Delta+m^2}-m$.}, and can be defined in several ways. For
example, we can associate to $\sqrt{-\Delta+m^2}$ its symbol
$\sqrt{k^2+m^2}$ in the following way: for any $f\in H^1(\R^N)$ with
Fourier transform ${\cal F} f$, we define
\[
{\cal F}(\sqrt{-\Delta+m^2} f)(k)=\sqrt{|k|^2+m^2}{\cal F}f(k),
\]
which is actually well defined in $H^{1/2}(\R^N)$, see \cite{LL} for
a complete description of this method.

We will follow another approach to define $\sqrt{-\Delta +m^2}$,
extending to the whole Euclidean space the ``Dirichlet to Neumann''
procedure (see, for example, \cite{CT}), which consists in realizing
the nonlocal operator $\sqrt{-\Delta+m^2}$ in $\R^3$ through a local
problem in $\R^4_+=\R^3\times(0,\infty)$. In view of our general
statements, we will present this procedure in every dimension
$N\geq2$: for any function $u\in {\mathscr S}(\R^N)$ there exists a
unique function $v\in {\mathscr S}(\RN)$ such that
\begin{equation}\label{soprax}
\begin{cases}
-\Delta v+m^2v=0 & \mbox{ in }\RN\\
v(x,0)=u(x) & \mbox{ on }\partial \RN=\R^N\times \{0\}\simeq \R^N,
\end{cases}
\end{equation}
i.e. $v$ is the generalized harmonic extension of $u$ in $\RN$.
Here, and in the following, we shall denote by $x$ a generic point
of $\R^N$. Now, consider the operator $T$ defined as
\begin{equation}\label{13}
Tu(x)=-\frac{\partial v}{\partial x_{N+1}}(x,0).
\end{equation}
It is readily seen that the system
\[
\begin{cases}
-\Delta w+m^2w=0 & \mbox{ in }\RN\\
w(x,0)=-\frac{\partial v}{\partial x_{N+1}}(x,0)=Tu(x) & \mbox{ on
}\partial \RN=\R^N\times \{0\}\simeq \R^N,
\end{cases}
\]
admits the unique solution $w(x,y)=-\frac{\partial v}{\partial
x_{N+1}}(x,y)$, and thus by \eqref{soprax} we immediately have
\[
T(Tu)(x)=-\frac{\partial w}{\partial x_{N+1}}(x,0)=\frac{\partial^2
v}{\partial x_{N+1}^2}(x,0)=(-\Delta v+m^2v)(x,0).
\]
In conclusion, $T^2=-\Delta+m^2$, i.e the
operator $T$ mapping the Dirichlet datum $u$ to the Neumann datum
$-\frac{\partial v}{\partial x_{N+1}}(\cdot,0)$ is a square root of the
generalized Laplacian $-\Delta+m^2$ in $\RN$.

Going back to \eqref{1FL}, such an equation does not take into
account the presence of outer influences or mutual interactions
among particles, so that it seems natural to include the presence of
a potential, considering
\begin{equation}\label{2FL}
i\psi_t=\sqrt{-\Delta+m^2}\,\psi-\lambda\big(W\ast |\psi|^2\big)\psi
+F'(\psi)\quad \mbox{ in }\R^N,
\end{equation}
where $\lambda\in\R$ and $F:\R\to \R$ is a $C^1$ potential having
some good invariant (common in Abelian Gauge Theories, see
\cite{bfcmp}, \cite{mug}, \cite{nuovo}): typically, one requires
some conditions of the form $F(e^{i\theta}u)=F(u)$ and
$F'(e^{i\theta}u)=e^{i\theta}F'(u)$ for any function $u$ and any
$\theta\in \R$, which is obviously satisfied by linear combinations
of power--like nonlinearities. Finally, $W$ is a radially symmetric
weight function which generalizes in $\R^N$ the Newtonian potential
$1/|x|$ in $\R^3$. In particular, we will also assume that $W$ can
be decomposed as sum of a bounded function plus another function
having suitable integrability (see Theorem \ref{main} for the
precise assumptions).

We are interested in solitary wave solutions of \eqref{2FL} of the
form
\[
\psi(x,t)=e^{-i\omega t}u(x), \mbox{ where $\omega\in \R$ and
$u:\R^N\to\R$};
\]
thus, since the Fourier transform acts only on the
$x$--variables, it is readily seen that $u$ solves
\[
\sqrt{-\Delta+m^2}u-\omega u-\lambda\big(W\ast u^2\big)u
+F'(u)=0\quad \mbox{ in }\R^N.
\]
Actually, there is no reason to discard an $x$--dependence in the
potential $F=F(x,s)$, so that we will actually consider the
following nonlinear Hartree equation with potential:
\begin{equation}\label{P}
\sqrt{-\Delta+m^2}u-\omega u-\lambda\big(W\ast u^2\big)u
+F_s(x,u)=0\quad \mbox{ in }\R^N.
\end{equation}
Let us remark that in $\R^3$ with $W(x)=1/|x|$ such an equation is
equivalent to the following system of pseudorelativistic
Schr\"odinger--Poisson type:
\begin{equation}\label{modello}
\begin{aligned}
\sqrt{-\Delta u+m^2}u-\omega u-\lambda u \phi+F_s(x,u)=0 &\quad \mbox{ in }\R^3,\\
-\Delta\phi=4 \pi u^2&\quad \mbox{ in }\R^3,
\end{aligned}
\end{equation}
where the second equation represents the repulsive character of the
Coulomb force (the attractive case is described by the equation
$\Delta\phi=4\pi u^2$, see \cite{ruarrso}). We also note that this
very last system is the pseudorelativistic version of the classical
Schr\"odinger--Poisson system, or Hartree--Fock equation, which has
been object of a large interest in the last decade, also for its
applications in modelling molecules and crystals, and we only quote
\cite{amru}, \cite{bftop}, \cite{cdl}, \cite{ce}, \cite{cale},
\cite{calbli}, \cite{tdnonex}, \cite{teadim}, \cite{MaZhao},
\cite{SPpos}, \cite{ru} and the papers cited therein as a brief
references list.

When $N=3$, the case $\lambda=1$ and $F\equiv0$ was studied in \cite{11} (for
existence of spherically symmetric solutions), in \cite{FJL} (for
stability of standing wave solutions of \eqref{1FL}), in \cite{FL}
(for instability of standing wave solutions of \eqref{1FL} when
$m=0$).

The fact that the natural domain for the governing operator in
\eqref{P} is $H^{1/2}(\R^N)$, forces to decrease the range of
natural exponents in the nonlinearities which could appear in the
equation, according to the Sobolev Embedding Theorem; more
precisely, if $N=3$, the critical exponent is 3; hence, all
subcritical superlinearities in the equation may have a growth
between 1 and 2. In fact, equation \eqref{P} was considered in
\cite{czn} when $\lambda=1$, $\omega<m$ and $m>0$, with
\begin{equation}\label{protW}
F(x,s)=F(s)=-\frac{|s|^p}{p}, \qquad p\in\left(2,\frac{2N}{N-1}\right).
\end{equation}

Let us remark that the prototype potential defined in \eqref{protW}
is not positive, while for physical reasons a potential suitable to
model physical phenomena should be nonnegative: indeed, the fact
that $F$ is nonnegative implies that the potential energy density of
a solution of equation \eqref{P} has more nonnegative contributions,
as one can see from the expression of the Lagrangian in equation
\eqref{J} below. Another reason to consider positive potential is
that, if we consider the classical autonomous electrostatic case
$-\Delta u+F'(u)=0$, calling ``rest mass'' of the particle $u$ the
quantity
$$
\int F(u)\,dx
$$
(see \cite{bfarc}), the fact that $F$ is positive implies that the
systems under consideration has - a priori - {\em positive mass},
which is, of course, relevant from a physical viewpoint.

Therefore, in this paper, as far as existence is concerned, we shall
consider equation \eqref{P} under the assumptions that $F$ is {\em
nonnegative} and $m>0$.

Using the approach with the operator $T$ introduced in \eqref{13},
we rewrite equation \eqref{P} as the following system:
\begin{equation}\label{PN}
\begin{cases}
-\Delta v+m^2v=0 &\mbox{ in }\R^{N+1}_+,\\
-\frac{\partial v}{\partial x_{N+1}}=\omega v+\lambda\big(W\ast
v^2\big)v -F_s(x,v)&\mbox{ on }\partial
\R^{N+1}_+=\R^N\times\{0\}\simeq \R^N.
\end{cases}
\end{equation}

As usual, for physical reasons, we look for solutions having finite
energy, i.e. $v\in H^1(\RN)$, $H^1(\RN)$ being the usual Sobolev space endowed with
the scalar product
\[
\langle v,w\rangle_{H^1}:=\int_{\RN}\big(D v\cdot D w+m^2vw\big)\,dxdx_{N+1}
\]
and norm $\|v\|=(\int |Dv|^2+m^2\int v^2)^{1/2}$.

\begin{defi}
A function $v\in H^1(\RN)$ is said to be a (weak) solution of
problem \eqref{PN} iff for every $w\in H^1(\RN)$ there holds
\begin{equation}\label{defsol}
\int_{\RN}(Dv\cdot Dw+m^2vw)\,dxdx_{N+1}=\int_{\R^N}\left[\omega
v+\lambda\big(W\ast v^2\big)v-F_s(x,v) \right]w\,dx.
\end{equation}
\end{defi}

\begin{rem}
We remark that, whenever a solution $v\in H^1(\RN)$ of \eqref{PN} is given, then
\[
\int_{\R^N}F_s(x,v)w\,dx\in \R \quad\mbox{ for all }w\in H^1(\RN)
\]
without assuming any growth condition on $F$.
\end{rem}

In this paper we are concerned with problem \eqref{PN} in presence
of a generally nontrivial potential $F$, always neglected in the
previous papers, except for \cite{czn}; hence, in order to cover the
trivial case $F\equiv 0$, for the main existence result (see Theorem
\ref{main}) we shall use the following  general superlinear and
subcritical assumptions, which, however, can be relaxed for the
other existence and for the non existence results, see Propositions
\ref{lambda0}, \ref{Lagrange} and \ref{genus} and Theorem
\ref{nonex} below:
\begin{description}
\item[$F_1)$] $F:\R^N\times \R\to [0,\infty)$ is such that the
derivative $F_s:\R^N\times \R\to \R$ is a Carath\'eodory function,
$F(x,s)=F(|x|,s)$ for a.e. $x\in \R^N$ and for every $s\in \R$, and
$F(x,0)=F_s(x,0)=0$ for a.e. $x\in \R^N$;
\item[$F_2)$] $\exists\,C_1,C_2>0$ and $2<\ell<p<2N/(N-1)$
such that
\[
\mbox{$|F_s(x,s)|\leq C_1|s|^{\ell-1}+C_2|s|^{p-1}$ for a.e. $x\in \R^N$ and every
$s\in \R$};
\]
\item[$F_3)$] $\exists\,k\geq 2 \mbox{ such that } 0\leq
sF_s(x,s)\leq kF(x,s)$  for a.e. $x\in \R^N$ and every
$s\in \R$.
\end{description}

\begin{rem}
Condition $F_3)$ is a kind of {\it reversed} Ambrosetti--Rabinowitz
condition, already used, for instance, in \cite{bf}, \cite{nuovo}
and \cite{SPpos}.
\end{rem}

\begin{rem}
As it will be clear from the existence proof, the requirement that $F$ depends
radially on the space variable is a technical assumption which lets
us reduce the problem to a radial setting and use some compactness
properties in the associated Sobolev space.
\end{rem}

\begin{rem}
From $F_1)$ we immediately get that $F$ has  an absolute minimum point for $s=0$ and
for a.e. $x\in \R^N$, so that problem \eqref{PN} always admits the
trivial solution $u=0$. Moreover, by direct integration of  $F_2)$ we get
\begin{equation}\label{crescitaW}
0\leq F(x,s)\leq \frac{C_1}{\ell}|s|^\ell+\frac{C_2}{p}|s|^p
\end{equation}
for every $s\in \R$ and a.e. $x\in \R^N$, so that  the potential $F$
is superquadratic at 0 and subcritical at infinity, as in the case
of \eqref{protW}. For this reason, in $F_3)$ we exclude the case
$k<2$, since by simple calculations we would get $F\equiv 0$,
already considered in the related literature.
\end{rem}

\begin{rem}
Our assumptions include the case $F\equiv 0$, but even the more
intriguing cases in which $F(x,s)=0$ for $x$ belonging to a proper
subset of $\R^N$ - a  ball (so that the potential $F$ is active only
in an exterior domain), an annulus, \ldots - or for some values of
$s$ - for instance when $s$ is large or small. All these situations
are completely new, and, to our best knowledge, this paper is the
first one to consider \eqref{PN} with a potential, even possibly
vanishing somewhere. In particular, we remark that no control from
below is assumed on the potential $F$, as done, for instance, in
\cite{bfcmp} for a Klein--Gordon--Maxwell system with positive
potential.
\end{rem}

Since $u=0$ is a solution of problem \eqref{PN}, we are interested
in nontrivial solutions. But, before giving our first existence
result, we make precise the assumptions on the weight potential $W$,
which will be assumed from now on:
\medskip

\noindent {\bf W)}: $W:\R^N\to [0,\infty)$ is such that
$W(x)=W(|x|)$, $W=W_1+W_2$, where $W_1\in L^r(\R^N)$ for some $r\in
(N/2,\infty)$, and $W_2\in L^\infty(\R^N)$.

\begin{rem}
If $N=3$ we can take the Newton potential $W(x)=1/|x|$, which is
bounded at infinity and in $L^r$ near the origin for $r<N$, so that
all our results cover the physical problem \eqref{modello}. However,
if we require that solutions have a constant $L^2$ norm equal to $M$
(like in \cite{11}), it turns out that the Newton potential is
critical, in the sense that if $W(x)=1/|x|$, then a radial,
real--valued, nonnegative ground state solitary wave in
$H^{1/2}(\R^3)$ does exist only if $M<M_c$, $M_c$ being the
Chandrasekhar limit mass for boson stars modelled by \eqref{1FL}
(and in such a case the solution is found via a minimization
process). Therefore, not fixing the $L^2$ norm of solutions leaves
us more freedom in the quest of solutions.

Moreover, another classical potential we can treat when $N=3$ is any
Yukawa type two body interaction, that is $W(x)=e^{-\mu|x|}/|x|$,
$\mu\geq0$.
\end{rem}

\begin{rem}
In contrast to \cite{czn} and to the behaviour of the Newton or
Yukava potential, we do not require that $W(x)\to 0$ as $|x|\to
\infty$, so that we are allowed to consider a wider class of
kernels.
\end{rem}

We are ready for our main existence result.
\begin{theo}\label{main}
Assume that $F$ satisfies $F_1)$, $F_2)$ and $F_3)$ with $k\leq 4$
and that $W$ satisfies {\bf W)}. Then for any $\la>0$ and $\omega<m$
there exists a nontrivial solution $v\in H^1(\RN)$ of problem
\eqref{PN} which is radially symmetric in the first $N$ variables.
If, in addition, $F(x,s)\geq F(x,|s|)$ for a.e. $x\in \R^N$ and all
$s\in \R$, then $v$ is strictly positive in $\overline{\RN}$.
\end{theo}

The assumption $\lambda>0$ in Theorem \ref{main} is not restrictive. Indeed,
just under a more general version of assumption $F_3)$ above, problem \eqref{PN}
admits only the trivial solution when $\lambda \leq 0$,
independently of any possible symmetry:
\begin{prop}\label{lambda0}
Suppose that $F_s(x,s)s\geq 0$ for a.e. $x\in \R^N$ and all $s\in
\R$, and let $v\in H^1(\RN)$ be a solution of problem \eqref{PN}
with $\lambda \leq 0$ and $\omega<m$. Then $v=0$.
\end{prop}

For the more general non existence results in $\R^3$ we assume that
$F$ satisfies a weaker form of $F_1)$, i.e. we require that $F$ is
independent of the $x$--variable and needs not be nonnegative:
\begin{description}
\item[$F_1)'$] $F:\R^3\to \R$ is of class $C^1$ with
$F(0)=F'(0)=0$.
\end{description}
The non existence results we shall prove are consequences of the
following variational identity involving the generalized half
Laplacian, which we consider of independent interest:
\begin{theo}\label{identita}
Assume $F_1)'$; if $v\in H^{1/2}(\R^3)$ is a solution of
\eqref{modello} such that
\[
\int_{\R^3}F(v)\,dx\in \R,
\]
then
\begin{equation}\label{nonso}
\begin{aligned}
0&=-\int_{\R^4_+}|Dv|^2dX-2m^2\int_{\R^4_+}v^2dX\\
&
+\frac{3}{2}\omega\int_{\R^3}v^2dx+\frac{5}{4}\lambda\int_{\R^3}v^2\phi\,dx-3\int_{\R^3}F(v)\,dx.
\end{aligned}
\end{equation}
\end{theo}

As aforesaid, by exploiting \eqref{nonso} we are able to prove
\begin{theo}\label{nonex}
Assume $F_1)'$; let $v\in H^{1/2}(\R^3)$ be a solution of
\eqref{modello} such that
\[
\int_{\R^3}F(v)\,dx\in \R.
\]
If
\begin{equation}\label{ipo1W}
\mbox{$\omega\leq 0$, $\lambda\leq 0$ and }F'(s)s\leq 3F(s) \qquad \forall\,s\in\R ,
\end{equation}
or
\begin{equation}\label{ipo2W}
\mbox{$\omega\leq0$, $\lambda\leq0$ and } 3F(s)\leq2F'(s)s \qquad \forall\,s\in\R ,
\end{equation}
or
\begin{equation}\label{ipo3W}
\mbox{$\lambda\leq0$, $\omega\in \left[0,m\sqrt{\frac{8}{9}}\right]$
and } F(s)\geq 0 \qquad \forall\,s\in\R ,
\end{equation}
or
\begin{equation}\label{ipo4W}
\mbox{$\lambda>0$, $\omega\in (0,m)$ and } 0\leq F'(s)s\leq 2F(s)
\qquad \forall\,s\in\R,
\end{equation}
then $v\equiv 0$.
\end{theo}
As a straightforward application, we have the following
\begin{cor}
If $v\in H^{1/2}(\R^3)\cap L^p(\R^3)$ solves \eqref{modello}, then
$v\equiv 0$ provided that
\begin{itemize}
\item $F(s)=\frac{|s|^p}{p}$ and $\left\{\begin{array}{l}
\omega\leq 0,\, \lambda\leq 0,\mbox{ and } p>0, \mbox{ or}\\
\lambda\leq 0,\, \omega\in \left[0,m\sqrt{\frac{8}{9}}\right]\mbox{ and } p>0,\mbox{ or}\\
\lambda>0,\, \omega\in(0,m),\, p\in(0,2]
\end{array}\right.$
\item $F(s)=-\frac{|s|^p}{p}$, $\omega\leq 0,$ $\lambda\leq 0$ and $p\in
\left(0,\frac{3}{2}\right]\cup [3,\infty)$.
\end{itemize}
\end{cor}

\begin{rem}
In condition \eqref{ipo3W}, contrary to the assumptions of
Proposition \ref{lambda0}, no sign assumption is made on $F'$.

The Corollary above shows that the existence result proved in
\cite{czn} for $\lambda>0$, $\omega<m$ and $F(s)=-|s|^p/p$ with
$2<p<2N/(N-1)$ is somehow optimal, since in the supercritical case
only the trivial solution shows up, independently of any possible
symmetry.

On the other hand, such a Corollary also guarantees that our
requirement of working with a superquadratic $F$ when $\lambda>0$
and $\omega<m$ is not only a technical one, since if $F$ is
subquadratic or quadratic ($p=2$), then only the trivial solution
can exist.
\end{rem}

Actually, Theorem \ref{nonex} is a special case of a more general
result, which is the following:
\begin{theo}\label{generale}
Assume $F_1)'$; let $v\in H^{1/2}(\R^3)$ be a solution of
\eqref{modello} such that
\[
\int_{\R^3}F(v)\,dx\in \R.
\]
If
\begin{equation}\label{ipo1Wrho}
\mbox{$\omega\leq 0$, $\lambda\leq 0$ and $\exists \, \rho\leq 1$
s.t. $\rho F'(s)s\leq 3F(s) \quad \forall\,s\in\R$},
\end{equation}
or
\begin{equation}\label{ipo2Wrho}
\mbox{$\omega\leq0$, $\lambda\leq0$ and $\exists \, \rho\geq 2$ s.t.
$3F(s)\leq \rho F'(s)s \quad \forall\,s\in\R$},
\end{equation}
or
\begin{equation}\label{ipo3Wrho}
\mbox{$\lambda\leq 0$, $\exists \, \rho>2$ s.t. $\omega\in
\left(0,2m\frac{\sqrt{(\rho-1)(\rho-2)}}{2\rho-3}\right]$ and
$3F(s)\leq \rho F'(s)s \ \forall\,s\in\R$},
\end{equation}
then $v\equiv 0$.
\end{theo}
We remark that the conditions in \eqref{ipo1Wrho} and
\eqref{ipo2Wrho} are not consequences of the corresponding ones in
\eqref{ipo1W} and \eqref{ipo2W}, since no assumptions on the sign of
$F$ or $F'$ is ever made.
\smallskip

We conclude our list of results with two existence statements which
can be proved without condition $F_3)$, but only assuming the
positivity and the subcriticality of $F$; in this case, $\lambda$
appears as a Lagrange multiplier:
\begin{prop}\label{Lagrange}
Assume $F_1)$, $F_2)$ and suppose that $\omega<m$. Then there exists
$\lambda\in \R$ such that the associated problem \eqref{PN} admits a
nontrivial solution $u\in H^1_r(\RN)$. If in addition, $2F(x,s)\leq
F_s(x,s)s$ for all $s\in \R$ and a.e. $x\in \R^N$, then $\lambda>0$.
\end{prop}

\begin{prop}\label{genus}
Assume $F_1)$, $F_2)$ and suppose that $\omega<m$. Moreover, suppose
that $F(x,s)=F(x,-s)$ for all $s\in\R$ and a.e $x\in \R^N$. Then
there exists a sequence $(\lambda_n)_n\in \R$ such that the
associated problems \eqref{PN} with $\lambda=\lambda_n$ admit a
couple of nontrivial solutions $\pm u_n\in H^1_r(\RN)$. If in
addition, $2F(x,s)\leq F_s(x,s)s$ for all $s\in \R$ and a.e. $x\in
\R^N$, then $\lambda_n>0$ for any $n\in \N$.
\end{prop}

From a physical viewpoint, positive solutions are the most
interesting ones, and in fact we can prove
\begin{prop}\label{u>0}
If $\lambda\leq 0$ or $F(x,s)\geq F(x,|s|)$ for a.e. $x\in \R^N$ and
all $s\in \R$, the solution found by Proposition $\ref{Lagrange}$ is
strictly positive in $\RN$.
\end{prop}

\begin{rem}
If only $F_1)$ and $F_2)$ are in force, the assumption
$\lambda\leq0$ does not contradict Proposition \ref{lambda0} above.
\end{rem}

\section{Proofs of the existence theorems}

Let us start making precise some notations:
\[
\begin{array}{ll}
(x,x_{N+1})  &\mbox{ a point of $\RN=\R^N\times (0,\infty)$},\\
\|v\| &\mbox{ the norm of $v$ in $H^1(\RN)$},\\
\|v\|_q  &\mbox{ the norm of $v$ in $L^q(\RN)$},\\
|v|_q  &\mbox{ the norm of (the trace of) $v$ in $L^q(\R^N)$}.
\end{array}
\]

Let us recall that any $v\in H^1(\RN)$, $N\geq 2$, admits trace
(still denoted by $v$ for simplicity) on $\partial
\RN=\R^N\times\{0\}\simeq\R^N$, and that the following embeddings
hold true (see \cite[Theorems 7.58 and 7.57]{adams}):
\[
H^1(\R^{N+1}_+)\hookrightarrow W^{\chi,q}(\R^N)\hookrightarrow  L^\sigma(\R^N),
\]
where $\chi=1-\frac{N+1}{2}+\frac{N}{q}\in[0,1)$ and $q\leq
\sigma\leq Nq/(N-\chi q)$. Let us note that $q=2N/(N-1+2\chi)\geq 2$
for any $\chi\leq 1/2$ and that $ Nq/(N-\chi q)=2N/(N-1)$ for any
$\chi$; in particular
\[
H^1(\R^{N+1}_+)\hookrightarrow L^2(\R^N) \quad \mbox { and } \quad
H^1(\R^{N+1}_+)\hookrightarrow L^{2N/(N-1)}(\R^N).
\]
In addition,
\[
\frac{4N}{N+2}< \frac{Nq}{N-\chi q} \quad \mbox{ for any $\chi\geq 0$},
\]
and
\[
\frac{4N}{N+2}\geq q
\]
as soon as $\chi\geq 1-N/4$. In particular, by interpolation, we get
that any $v\in H^1(\R^{N+1}_+)$ is such that $v^2\in
(L^{2^\ast}(\R^N))'=L^{\frac{2N}{N+2}}(\R^N)$.

Moreover, if $v\in C^\infty(\RN)$, we have
\[
\begin{aligned}
\int_{\R^N}|v(x,0)|^qdx&=-\int_{\R^N}dx\int_{0}^\infty
\frac{\partial}{\partial x_{N+1}}|v(x,x_{N+1})|^qdx_{N+1}\\ & =-
q\int_{\RN}|v(x,x_{N+1})|^{q-2}v(x,x_{N+1})\frac{\partial v}{\partial
x_{N+1}}(x,x_{N+1})\,dxdx_{N+1},
\end{aligned}
\]
and by the H\"older inequality
\begin{equation}\label{trace}
|v(x,0)|_q\leq q^{1/q}\|v\|_{2q-2}^{1-1/q}\|v_{x_{N+1}}\|_2^{1/q}\leq q^{1/q}\|v\|_{2q-2}^{1-1/q}\|Dv\|_2^{1/q}.
\end{equation}
By interpolation, the Sobolev inequality and by density, we get that
\begin{equation}\label{traceh}
|v(x,0)|_q\leq c_q\|v\| \quad \mbox{ for any }v\in H^1(\RN),
\end{equation}
provided that $2\leq 2q-2\leq 2(N+1)/(N-1)$, that is $2\leq q\leq 2N/(N-1)$.

Applying the Young inequality to \eqref{trace} we also get
\begin{equation}\label{Cp}
|v(x,0)|_q^q\leq \frac{\ve q^2}{4}
\int_{\RN}|v|^{2q-2}dxdx_{N+1}+\frac{1}{\ve}
\int_{\RN}\left|\frac{\partial v}{\partial x_{N+1}}\right|^2dxdx_{N+1}
\end{equation}
for any $\ve>0$, and in particular, when $q=2$, we have
\begin{equation}\label{C2}
|v(x,0)|_2^2\leq \ve \int_{\RN}|v|^2dxdx_{N+1}+\frac{1}{\ve}
\int_{\RN}\left|\frac{\partial v}{\partial x_{N+1}}\right|^2dxdx_{N+1} \quad
\mbox{ for any $\ve>0$},
\end{equation}

Now, write
\[
\int_{\R^N}\big(W\ast v^2\big)v^2(x)\,dx=\int_{\R^N}\big(W_1\ast
v^2\big)v^2(x)\,dx+\int_{\R^N}\big(W_2\ast v^2\big)v^2(x)\,dx.
\]
By H\"older's inequality, we can estimate
the right hand side of the previous inequality by
\[
|v|_{2q}^2|W_1\ast v^2|_{q'}+|W_2|_\infty|v|^4_2.
\]
Now apply Young's inequality for convolutions choosing $q$ so that
$1/q'=1/r+1/q-1$, that is $q=2r/(2r-1)$, estimating with
\[
|v|_{2q}^2|W_1|_r|v^2|_{q} +|W_2|_\infty|v|^4_2;
\]
note that, since $r>N/2$, we have $1<q<N/N-1$. Finally, by the
interpolation and the Sobolev inequalities, we get that there exists
$C=C(W)>0$ such that
\begin{equation}\label{maggconv}
\int_{\R^N}\big(W\ast v^2\big)v^2(x)\,dx \leq C\|v\|^4 \quad \mbox{
for any }v\in H^1(\RN).
\end{equation}

In view of the previous remarks, by exploiting the radial symmetry
of $W$, the proof of the following result is straightforward:
\begin{prop}
A function $v\in H^1(\RN)$ is a solution of system \eqref{PN} iff
$v$ is a critical point of the $C^1$ functional $J:H^1(\RN)\to\R$
defined as
\begin{equation}\label{J}
\begin{aligned}
J(v)=&\frac{1}{2}\int_{\RN}(|Dv|^2+m^2v^2)\,dxdx_{N+1}\\
&-\int_{\R^N} \left[\frac{\omega}{2}v^2+\lambda \big(W\ast
v^2\big)v^2-F(x,v)\right]dx.
\end{aligned}
\end{equation}
\end{prop}

Due to the lack of compactness of the rotation group in $\R^N$, we
will look for critical points of $J$ constrained on the space of
functions which are radially symmetric in the first $N$ variables,
that is
\[
H^1_r(\RN)=\Big\{v\in H^1(\RN) \, :\, v(Mx,x_{N+1})=v(x,x_{N+1})\mbox{ for
any }M\in O(N)\Big\};
\]
here $O(N)$ denotes the orthogonal group in $\R^N$. Since the
problem under consideration is invariant by rotation around the
$x_{N+1}$--axis, if $v\in H^1_r(\RN)$ is a critical point of $J$
constrained on $H^1_r(\RN)$, then $v$ is also a critical point of
$J$ on the whole of $H^1(\RN)$ by the Principle of Symmetric
Criticality of Palais, see \cite{PA}. Hence, we will look for
critical points of $J$ constrained on $H^1_r(\RN)$. In particular,
we want to show that under the assumptions of Theorem \ref{main},
the functional $J$ constrained on $H^1_r(\RN)$ satisfies the
hypothesis of the Mountain Pass Theorem (see \cite{ar}).

First, we prove that $J_{|H^1_r(\RN)}$ has a strict minimum point in
$v=0$. Indeed, taking $\ve=m$ in \eqref{C2}, since $F\geq 0$, using
\eqref{maggconv}, we have
\[
J(v)\geq \left(\frac{1}{2}-\frac{\omega}{2m}\right)
\int_{\RN}|Dv|^2dxdx_{N+1} +
\left(\frac{m^2}{2}-\frac{m\omega}{2}\right) \int_{\RN}v^2dxdy
-C\|v\|^4
\]
for any $v\in H^1_r(\RN)$, and the claim follows, since $\omega<m$.

Moreover, if $v\in H^1_r(\RN)$ is such that $v(x,0)\not \equiv 0$, taken
$t>0$ we have
\[
\begin{aligned}
J(tv)&=\frac{t^2}{2}\int_{\RN}(|Dv|^2+m^2v^2)\,dxdx_{N+1}\\
& -\int_{\R^N} \left[\frac{\omega t^2}{2}v^2+\lambda t^4\big(W\ast
v^2\big)v^2-F(x,tv)\right]dx\\
& \leq At^2 -Bt^4 +Ct^{\ell}+Dt^p
\end{aligned}
\]
by \eqref{crescitaW}, where $A,B,C,D$ are positive constants. Now,
if $p<4$, we have that $At^2 -Bt^4 +Ct^{\ell}+Dt^p\to -\infty$ as
$t\to \infty$. Note that the requirement $p<4$ is automatically
satisfied, since $p<\frac{2N}{N-1}<4$ for every $N\geq2$.

We also need
\begin{lem}\label{ps}
$J$ satisfies the Palais--Smale condition on $H^1_r(\RN)$, i.e.,
every sequence $(v_n)_n$ in $H^1_r(\RN)$ such that $(J(v_n))_n$ is
bounded and $J'(u_n)\to 0$ in $(H^1_r(\RN))'$ as $n\to \infty$,
admits a convergent subsequent.
\end{lem}
\begin{proof}
Let $(v_n)_n$ in $H^1_r(\RN)$ be as in the statement above. We
first prove that $(v_n)_n$ is bounded. Indeed, by assumption, there
exist $A,B>0$ such that
\[
4J(v_n)-J'(v_n)(v_n)\leq A+B\|v_n\|.
\]
On the other hand,
\[
\begin{aligned}
4J(v_n)-J'(v_n)(v_n)& =\int_{\RN}\big[|Dv_n|^2+m^2v_n^2\big]
\,dxdx_{N+1}-\omega\int_{\R^N}v_n^2dx\\
&+\int_{\R^N}\big[4F(x,v_n)-F_s(x,v_n)v_n\big]\, dx\\
&\geq \left(1-\frac{\omega}{m}\right)\int_{\RN}|Dv_n|^2dxdx_{N+1}+m
(m-\omega)\int_{\RN}v_n^2dxdx_{N+1}
\end{aligned}
\]
by \eqref{C2} applied again with $\ve=m$ and by $F_3)$. In
conclusion, there exist $C>0$ such that
\[
C\|v_n\|^2\leq A+B\|v_n\|,
\]
and thus
\begin{equation}\label{limitat}
\mbox{$(v_n)_n$ is bounded in $H^1_r(\RN)$},
\end{equation}
as claimed, and so we can assume that
$v_n\rightharpoonup v$ in $H^1_r(\RN)$.

By \eqref{traceh}, (denoting as usual the trace of a function by the
function itself) we get that $(v_n)_n$ is bounded also in
$L^q(\R^N)$ for any $q\in\left[2,\frac{2N}{N-1}\right]$, so that we
may assume, without loss of generality, that $v_n\rightharpoonup v$
in $L^q(\R^N)$. However, by the compact embedding of
$H^{1/2}_r(\R^N)$ in $L^q_r(\R^N)$ for any
$q\in\left(2,\frac{2N}{N-1}\right)$, see  \cite{ss}, we actually
have that
\begin{equation}\label{convergenzaforte}
\mbox{$v_n\to v$ in $L^q(\R^N)$ for any $q\in\left(2,\frac{2N}{N-1}\right)$.}
\end{equation}
Up to subsequence, we can also assume that $v_n\to v$ a.e. in $\RN$.

We will now show that $v_n\to v$ in $H^1_r(\RN)$. First,
$F_s(x,v_n)\to F_s(x,v)$ a.e., since $F_s$ is a Carath\'eodory
function by $F_1)$, and by $F_2)$
\[
|F_s(x,v_n)(v_n-v)|\leq C_1|v_n|^{\ell-1}|v_n-v|+C_2|v_n|^{p-1}|v_n-v|.
\]
For instance,
\[
\int_{\R^N}|v_n|^{\ell-1}|v_n-v| \,dx\leq |v_n|_\ell^{\ell-1}
|v_n-v|_\ell\to 0  \quad \mbox{ as }n\to \infty
\]
by \eqref{convergenzaforte} and \eqref{limitat}. Analogously, we have
\[
|F_s(x,v)(v_n-v)|\leq C_1|v|^{\ell-1}|v_n-v|+C_2|v|^{p-1}|v_n-v|\to
0  \quad \mbox{ as }n\to \infty,
\]
and so  Lebesgue's Theorem gives
\begin{equation}\label{magwn}
\int_{\R^N}[F_s(x,v_n)-F_s(x,v)](v_n-v)\,dx \to 0 \quad \quad \mbox{
as }n\to \infty.
\end{equation}

Moreover,  for any $q\in\left(2,\frac{2N}{N-1}\right)$, so that
$q'\in \left(\frac{2N}{N+1},2\right)$, we get
\begin{equation}\label{magcolq}
\begin{aligned}
\bigg| & \int_{\R^N}\big(W\ast v_n^2\big)v_n(v_n-v)dx\bigg|\\
& \leq |v_n-v|_q \left(\int_{\R^N}|v_n(z)|^{q'}\left(
\int_{\R^N}W(x-z)v_n^2(x)dx\right)^{q'}dz \right)^{1/q'}.
\end{aligned}
\end{equation}

First, Minkowski's inequality implies that for any
$\wp\in[1,\infty]$ there holds
\begin{equation}\label{minko}
|W\ast v_n^2|_{\wp}\leq |W_1\ast v_n^2|_{\wp}+|W_2\ast v_n^2|_{\wp}.
\end{equation}

Now, by H\"older's inequality and  Young's inequality for
convolutions, for any $t\in (1,\infty)$ we have
\begin{equation}\label{miserve}
\begin{aligned}
\int_{\R^N} |v_n|^{q'}\left( \int_{\R^N}W_1(x-z)v_n^2(x)dx\right)^{q'}dz & \leq
|v_n|_{q't}^{q'}|W_1\ast v_n^2|_{t'q'}^{q'} \\
& \leq |v_n|_{q't}^{q'}|W_1|_r^{q'}|v_n^2|_{\sigma}^{q'},
\end{aligned}
\end{equation}
where $1+\frac{1}{t'q'}=\frac{1}{r}+\frac{1}{\sigma}$. Choosing $t$
such that $q't=2\sigma$, that is $\sigma=\frac{3q'r}{2(q'r+r-q')}$,
we finally have
\begin{equation}\label{magcolq'}
|v_n(W_1\ast v_n^2)|_{q'}\leq |v_n|_{2\sigma}^3|W_1|_r.
\end{equation}
Requiring $2\sigma \in\left[2, \frac{2N}{N-1}\right]$ implies
\[
\frac{2r}{r+2}\leq q' \leq \frac{2Nr}{rN-3r+2N}.
\]
Since $r>N/2$, we have
\[
\frac{2N}{N+4}<\frac{2r}{r+2}<2 \mbox{ and
}\frac{2N}{N+1}<\frac{2Nr}{rN-3r+2N}< \frac{2N}{N-3};
\]
finally have the nontrivial possibility $q'\in
\left(\frac{2N}{N+1},2\right)$, that is $q\in
\left(2,\frac{2N}{N-1}\right)$, with obvious meaning if $N\leq 3$.

In this way \eqref{magcolq}, \eqref{magcolq'} and \eqref{limitat}
imply the existence of $C>0$ such that
\begin{equation}\label{mag1}
\int_{\R^N} |v_n|^{q'}\left(
\int_{\R^N}W_1(x-z)v_n^2(x)dx\right)^{q'}dz\leq C \mbox{ for every
$n\in \N$}.
\end{equation}

On the other hand, proceeding as above, we have
\[
|v_n(W_2\ast v_n^2)|_{q'}\leq |v_n|_{q't}|v_n|_{2\sigma}^2|W_2|_\infty,
\]
where $t$ can be chosen so that
$q't=2\sigma\in\left[2,\frac{2N}{N-1}\right]$, obtaining, as above,
\begin{equation}\label{mag5}
|v_n(W_2\ast v_n^2)|_{q'}\leq |v_n|_{q't}|v_n|_{2\sigma}^2|W_2|_\infty\leq C
\end{equation}
for some $C>0$ and all $n\in \N$.

In conclusion, from \eqref{minko} we get
\begin{equation}\label{quasifinale}
\left|\int_{\R^N}\big(W\ast v_n^2\big)v_n(v_n-v)dx\right|\leq
C|v_n-v|_q
\end{equation}
for some $C>0$ and all $n\in \N$.

Reasoning as above, one can also prove that
\begin{equation}\label{quasifinale1}
\left|\int_{\R^N}\left(W\ast v^2\right)v(v_n-v)dx\right| \leq C
|v_n-v|_q
\end{equation}
for some $C>0$ and all $n\in \N$.

We also need to prove that $v$ is a critical point of $J$ on
$H^1_r(\RN)$. To this purpose, start from $J'(v_n)(w)\to 0$ as $n\to
\infty$ for every $w\in H^1_r(\RN)$, i.e.
\[
\langle v_n,w\rangle _{H^1}-\int_{\R^N} \big[\omega v_nw+\lambda
W\ast v_n^2 w-F_s(x,v_n)w\big]\,dx \to 0.
\]
Thus, by the weak convergence of $v_n$ to $v$, it is enough to prove
that
\[
\lim_{n\to \infty}\int_{\R^N} \big[\lambda \big(W\ast
v_n^2\big)v_nw- F_s(x,v_n)w\big]\,dx= \int_{\R^N} \big[\lambda
\big(W\ast v^2\big) vw-F_s(x,v)w\big]\,dx.
\]
First, by $F_2)$ we have
\[
\begin{aligned}
\left|\int_{\R^N}F_s(x,v_n)w \,dx\right|&\leq
\int_{\R^N}\big[C_1|v_n|^{\ell-1} +C_2|v_n|^{p-1}\big]|w|\,dx \mbox{ (by H\"older's inequality)} \\
&\leq
C_1|v_n|_{\ell}^{\ell-1}|w|_{\ell}+C_2|v_n|_{p}^{p-1}|w|_{p}\to
C_1|v|_{\ell}^{\ell-1}|w|_{\ell}+C_2|v|_{p}^{p-1}|w|_{p}.
\end{aligned}
\]
Hence, by the Generalized Dominated Convergence Theorem
\[
\lim_{n\to \infty}\int_{\R^N}F_s(x,v_n)w\,dx=
\int_{\R^N}F_s(x,v)w\,dx.
\]
Then,
\[
\begin{aligned}
& \left|\int_{\R^N}\Big[ \big(W\ast v_n^2\big)v_nw-\big(W\ast
v^2\big) vw\Big]\,dx\right|\\ &=\left|\int_{\R^N}\big(W\ast
v_n^2\big)(v_n-v)w\,dx+\int_{\R^N}\big(W\ast(v_n^2-v^2)vw\,dx\right|,
\end{aligned}
\]
and starting like in \eqref{miserve}, we can estimate the previous
identity with
\[
\big(|W_1|_r+|W_2|_\infty\big)\big(|v_n-v|_{2\sigma}|v_n|_{2\sigma}
+|v_n^2-v^2_\sigma|v|_{2\sigma} \big)|w|_{2\sigma},
\]
where $2\sigma \in\left(2,\frac{2N}{N-1}\right)$. Hence, proceeding
as done to obtain \eqref{quasifinale} and \eqref{quasifinale1}, the
last quantity goes to 0 as $n\to \infty$ and thus we can conclude
that $v$ is a critical point of $J$ on $H^1_r(\RN)$.

Finally, we now know that $J'(v_n)(v_n-v)=[J'(v_n)-J'(v)](v_n-v)\to
0$ as $n\to\infty$, i.e.
\[
\begin{aligned}
&\|v_n-v\|^2-\omega \int_{\R^N}(v_n-v)^2\,dx-\lambda \int_{\R^N}
\left(v_n(W\ast v_n^2)-v(W\ast v^2)\right)(v_n-v)\,dx\\
&+\int_{\R^N}[F_s(x,v_n)-F_s(x,v)](v_n-v)\,dx\to 0 \quad\mbox{ as }n\to\infty.
\end{aligned}
\]

By \eqref{C2} applied again with $\ve=m$, from \eqref{magwn},
\eqref{quasifinale}, \eqref{quasifinale1} and
\eqref{convergenzaforte}, we have
\begin{equation}\label{finePS}
\begin{aligned}
o(1)&=J'(v_n)(v_n-v)\\
&\geq
\left(1-\frac{\omega}{m}\right)\int_{\RN}|D(v_n-v)|^2dxdx_{N+1}+
m(m-\omega)\int_{\RN}(v_n-v)^2dxdx_{N+1}
\end{aligned}
\end{equation}
where $o(1)\to 0$ as $n\to \infty$. Since $\omega<m$, we finally get
that $v_n\to v$ in $H^1_r(\RN)$, as claimed.
\end{proof}

At this point, the existence of a nontrivial critical point $v$ for
$J$ is proved by applying the Mountain Pass Theorem, and so the
first part of Theorem \ref{main} holds true.

Now, assume also that $F(x,s)\geq F(x,|s|)$. Recall that for the
critical point found above, the critical level $\beta=J(v)$ is
defined as
\[
\beta=\inf_{\gamma\in \Gamma}\max_{t\in[0,1]}J(\gamma(t)),
\]
where $\Gamma=\{\gamma \in C([0,1],H^1_r(\RN)\,:\,\gamma(0)=0,
J(\gamma(1))<0\}$. Now, the map $t\mapsto J(tv)$ is increasing and
has a strict maximum point at $t=1$; taken $\tau>0$ such that
$J(\tau v)<0$, define $\tilde \gamma(t)=t\tau |v|$, so that $\tilde
\gamma \in \Gamma$, by the additional assumption on $F$; moreover,
the map $t\mapsto J(t|v|)$ has a unique maximum point at $t=1$.
Thus, it is readily seen that $J(t\tau v)\geq J(\tilde \gamma(t))$,
so that
\[
\beta=\max_{t\in[0,1]}J(t\tau v)\geq \max_{t\in[0,1]}J(\tilde \gamma(t))\geq \beta.
\]
Hence, also $t\mapsto t|v|$ is a path giving the same critical
level. If, by contradiction, $|v|$ were not a critical point for
$J$, we could define a path $\bar\gamma\in\Gamma$, obtained
deforming $\tilde \gamma$ via the gradient flow, in such a way that
$\max_{t\in[0,1]}J(\bar\gamma(t))<\beta$, contradicting the
definition of $\beta$ itself. In conclusion, also $|v|$ is a
critical point for $J$.

Finally, if there exists $(\bar x,\bar x_{N+1})\in \RN$ such that
$v((\bar x,\bar x_{N+1})=0$, by the maximum principle for the
equation in $\RN$, we get $x_{N+1}=0$ and by the boundary condition
we have $\frac{\partial v}{\partial {x_{N+1}}}(\bar x,0)=0$, in
contradiction with Hopf's Lemma. Hence $v>0$ and Theorem \ref{main}
is completely proved.

\begin{rem}
In \cite{czn} the proof of the Palais--Smale condition for the
related problem is given without using the compact embedding of
\cite{ss}, but exploiting the assumption that $W(x)\to 0$ as $|x|\to
\infty$, which here we do not require.
\end{rem}

\medskip

We now give the
\begin{proof}[Proof of Proposition $\ref{Lagrange}$]
Since $\omega<m$, the quantity
\[
\frac{1}{2}\int_{\RN}(|Dv|^2+m^2v^2)\,dxdy-\int_{\R^N}\frac{\omega}{2}v^2dx
\]
defines a norm $\|v\|^2_{\sim}$ in $H^1_r(\RN)$ which is equivalent
to the usual one. Indeed, applying \eqref{C2} with $\ve=m$, we get
\[
\|v\|^2_{\sim}\geq
\left(1-\frac{\omega}{m}\right)\int_{\RN}|Dv|^2dxdy +(m^2-\omega
m)\int_{\RN}v^2dxdy\geq C\|v\|^2
\]
for any $v\in H^1_r(\RN)$.

Now, let us set
\[
M=\left\{v\in H^1_r(\RN)\,:\, \int_{\R^N}v^2\big(W\ast
v^2\big)\,dx=1\right\}={\mathscr G}^{-1}(0),
\]
where ${\mathscr G}(v)= \int_{\R^N}v^2\big(W(x)\ast v^2\big)\,dx-1$.

It is easy to see that $M$ is a non empty differentiable manifold of
codimension 1. Indeed, for any $v\in H^1_r(\RN)$ such that
$v(x,0)\not=0$, consider the map $\psi:[0,\infty)\to \R$ defined as
\begin{equation}\label{psimia}
\psi(t)=t \int_{\R^N}v^2\big(W\ast v^2\big)\,dx.
\end{equation}
It is clear that $\psi$ is a strictly increasing function such that
$\psi(0)=0$ and $\psi(t)\to\infty$ as $t\to \infty$, and thus $M$ is
non empty. Moreover, if $v\in M$ and we assume that ${\mathscr
G}'(v)(w)=0$, i.e. (using the symmetry of $W$)
\[
\int_{\R^N}\big(W\ast v^2\big)vw\,dx=0 \quad \forall\,w\in
H^1_r(\RN),
\]
we get in particular
\[
\int_{\R^N}v^2\left(W\ast v^2\right)\,dx=0,
\]
a contradiction with the fact that $v\in M$.

Now, since $F\geq 0$, we immediately see that the $C^1$ functional
$I:H^1_r(\RN)\to \R$ defined as
\[
I(v)=\frac{1}{2}\|v\|^2_{\sim}+\int_{\R^N}F(x,v)\,dx
\]
is bounded below. Let $(v_n)_n\subset M$ be a minimizing sequence
for $I$ on $M$. It is readily seen that $(v_n)_n$ is bounded, so
that we may assume that $v_n\rightharpoonup v$ in $H^1_r(\RN)$,
$v_n\to v$ in $L^q(\R^N)$ for every $q\in
\left(2,\frac{2N}{N-1}\right)$ and a.e. in $\R^N$. Moreover, by
Ekeland's Variational Principle (see for example \cite[Theorem
8.5]{willem}) we can also assume that $I_{|M}'(v_n)\to 0$, i.e.
there exists a sequence $(\mu_n)_n$ in $\R$ such that
\begin{equation}\label{quasiLag}
I'(v_n)(w)-\mu_n\int_{\R^N}\left(W\ast v_n^2\right)v_nw\,dx\to 0
\end{equation}
as $n\to \infty$ for every $w\in M$, and hence for any $w\in
H^1_r(\RN)$. Since $v_n\in M$ for any $n\geq1$ and $(v_n)_n$ is
bounded, from $I'(v_n)(v_n)-\mu_n\to 0$, i.e.
\[
\|v_n\|_{\sim}^2-\int_{\R^N}F_s(x,v_n,v_n)\,dx-\mu_n\to0,
\]
we get that also $\int_{\R^N}F_s(x,v_n,v_n)\,dx-\mu_n$ is bounded.
By $F_2)$ we have
\[
\left|\int_{\R^N}F_s(x,v_n,v_n)\,dx\right|\leq
C_1\int_{\R^N}|v_n|^{\ell}dx+C_2\int_{\R^N}|v_n|^pdx\leq C
\]
for some universal constant $C>0$, since $(v_n)_n$ is bounded. In
conclusion, also $(\mu_n)_n$ is bounded. Hence, we can suppose that
there exists $\lambda\in \R$ such that $\mu_n\to \lambda$ as $n\to
\infty$. Proceeding as in the proof of Lemma \ref{ps}, we can now
show that, up to subsequences, $v_n\to v$ in $H^1_r(\RN)$. Moreover,
$M$ being closed in $H^1_r(\RN)$, we get that $v\in M$ is a
nontrivial minimum point of $I$ in $M$. Passing to the limit in
\eqref{quasiLag}, we get
\[
I'(v)(w)-\lambda\int_{\R^N}\left(W\ast v^2\right)vw\,dx=0
\]
for all $w\in M$ - hence for all $w\in H^1_r(\RN)$ -, i.e. $v$
solves \eqref{PN} with $\lambda$ given as a Lagrange multiplier.

If in addition $2F(x,s)\leq F_s(x,s)s$ for all $s\in \R$ and a.e. $x\in \R^N$, then
\[
\begin{aligned}
\lambda\int_{\R^N}\left(W\ast
v^2\right)v^2dx&=I'(v)(v)=\|v\|^2_{\sim}+
\int_{\R^N}F_s(x,v)v\,dx\\
&=2I(v)+\int_{\R^N}[F_s(x,v)v-2F(x,v)]\,dx\geq 2I(v)>0,
\end{aligned}
\]
and so $\lambda>0$.
\end{proof}

\begin{proof}[Proof of Proposition $\ref{genus}$]
It is an application of the following multiplicity theorem based on
the Krasnoselskii genus, see \cite[Theorem 5.7]{Struwe}:
\begin{theo}\label{sopra}
Suppose that $I$ is an even $C^1$ functional on a complete symmetric
$C^{1,1}$ manifold $\mathcal M$ contained in a Banach space $B$, and
suppose that $I$ satisfies $(PS)$ and is bounded from below. Let
\[
\tilde{\gamma}(\mathcal M)=\sup\{\gamma(K)\,:\,K\subset \mathcal
M\mbox{ is compact and symmetric}\}\leq \infty.
\]
Then $I$ admits at least $\tilde{\gamma}(\mathcal M)$ pairs of
critical points on $\mathcal M$.
\end{theo}
Here $\gamma(A)$ denotes the Krasnoselskii genus of a symmetric set
$A$, defined as $\gamma(\emptyset)=0$, and when $A\neq \emptyset$,
\[
\gamma(A)=\left\{\begin{array}{l}
\inf\{m\in \N\,:\,\exists \,h\in
C^0(A,\R^m\setminus\{0\}) \mbox{ odd}\},\\
\infty \mbox{ if }\{m\in \N\,:\,\exists \,h\in
C^0(A,\R^m\setminus\{0\}) \mbox{ odd}\}=\emptyset.
\end{array}\right.
\]
In particular $\gamma(A)=\infty$ for any symmetric set containing 0;
see \cite{Rab} or \cite{Struwe} for an introduction to the genus and
some related results and applications.

In our case, we proceed as in the previous proof, restricting $I$ on
the symmetric manifold $M$. Imitating the steps above, one can see
that $I$ satisfies the $(PS)$ condition on $M$, while, applying
Theorem \ref{sopra}, the existence part of the Proposition follows
from the following Lemma, whose proof is given in the Appendix:
\begin{lem}\label{genere}
$\tilde{\gamma}(M)=\infty$.
\end{lem}
The final statement in Proposition \ref{genus} is exactly as in the
previous proof.
\end{proof}

We conclude this section with the
\begin{proof}[Proof of Proposition $\ref{u>0}$]
If $\lambda\leq0$, it is enough to repeat the final part of the
proof of Proposition \ref{Lagrange} replacing the functional $I$ by
the functional
\[
{\cal I}(v)=\frac{1}{2}\|v\|^2_{\sim}+\int_{\R^N}F(x,v^+)\,dx,
\]
where $v^+=\max\{v,0\}$, so that the minimum point solves
\[
\begin{cases}
-\Delta v+m^2v=0 &\mbox{ in }\R^{N+1}_+,\\
-\frac{\partial v}{\partial x_{N+1}}=\omega v+\lambda\left(W\ast
v^2\right)v -F_s(x,v^+)&\mbox{ on }\R^N\times\{0\}=\partial
\R^{N+1}_+.
\end{cases}
\]
Using $v^-$ as test function, we find
\[
\begin{aligned}
-\|v^-\|_{\sim}^2&=-\lambda \int_{\R^N} \left(W\ast
v^2\right)(v^-)^2dx- \int_{\R^N}F_s(x,v^+)v^-dx\\ &=-\lambda
\int_{\R^N} \left(W\ast v^2\right)(v^-)^2dx\geq 0,
\end{aligned}
\]
so that $v=v^+\geq0$.

If $F(x,s)\geq F(x,|s|)$ one can immediately see that $I(v)\geq
I(|v|)$, so that $|v|$ is again a minimum point for $I$.

The strong maximum principle implies $v>0$ in  both cases.
\end{proof}

\section{Variational identities and proof of the non existence results}

We start this section with the first non existence result, whose
proof is very easy and can be obtained without additional new tools.
In particular, here we don't need the new variational identities for
the half Laplacian (see Lemma \ref{lem4} and equation
\eqref{nonex}), which will be developed below in order to prove the
more general non existence results.
\begin{proof}[Proof of Proposition $\ref{lambda0}$]
Taking $v$ as test function in \eqref{defsol}, if $\omega\leq 0$ we have
\[
\begin{aligned}
0&=\int_{\RN}(|Dv|^2+m^2v^2)\,dxdx_{N+1}-\omega\int_{\R^N}v^2
dx-\lambda\int_{\R^N}\left(W\ast v^2\right)v^2dx\\
&+ \int_{\R^N}F_s(x,v)v\,dx \geq \int_{\RN}(|Dv|^2+m^2v^2)\,dxdx_{N+1},
\end{aligned}
\]
and the thesis follows. If $\omega\in(0,m)$, applying \eqref{C2} with $\ve=m$, we find
\[
\begin{aligned}
0&=\int_{\RN}(|Dv|^2+m^2v^2)\,dxdx_{N+1}-\omega\int_{\R^N}v^2
dx-\lambda\int_{\R^N}\left(W\ast v^2\right)v^2dx\\
&+ \int_{\R^N}F_s(x,v)v\,dx \\
&\geq \left(1-\frac{\omega}{m}\right)\int_{\RN}|Dv|^2dxdx_{N+1}
+(m^2-\omega m)\int_{\RN}v^2dxdx_{N+1}\geq C\|v\|^2
\end{aligned}
\]
for some $C>0$, and again the claim is proved.
\end{proof}

Now, we show by some variational identities that the existence
results of the previous section (in particular Theorem \ref{main})
are, in some sense, optimal, provided that $F=F(s)$.

Let $v\in H^1(\R^{N+1}_+)$ be a solution of \eqref{PN} with
$F=F(s)$. By reasoning as in \cite[Theorem 3.2 and Proposition
3.9]{czn}, we can show that
\begin{itemize}
\item $v\in L^\infty(\R^{N+1}_+)$,
\item $v\in L^p(\R^N)$ $\forall\,p\in[2,\infty]$,
\item $v\in C^{0,\alpha}(\overline{\R^{N+1}_+})\cap W^{1,q}(\R^N\times(0,R))$ for any
$q\in[2,\infty)$ and all $R>0$,
\item if $F$ is of class $C^{0,\alpha}(\R)$, then $v\in C^{1,\alpha}(\overline{\R^{N+1}_+})\cap
C^2(\R^{N+1}_+)$ and is a classical solution of \eqref{PN}.
\end{itemize}

Let us set $X=(x,x_{N+1})$ with $x=(x_1,\ldots,x_N)\in \R^N$.
Moreover, define
\[
\Delta_R=\Big\{X\in \R^N\times [0,\infty)\,:\,|X|^2\leq R\Big\},
\]
\[
S_R^+=\Big\{X\in \R^{N+1}_+\,:\,|X|=R \mbox{ and }x_{N+1}>0\Big\}
\]
and
\[
b_r=\Big\{X\in \Delta_R\,:\,x_{N+1}=0 \Big\}.
\]
Finally, for shortness, we write $v_i$ in place of $v_{x_i}$.

We state the following result for regular functions, like solutions
of \eqref{PN} if $F$ is of class $C^{0,\alpha}$, with obvious
generalization for Sobolev functions in $H^2_{\rm loc}(\RN)\cap
H^1(\RN)$, since functions n this space admit traces on every
manifold appearing in the calculations below.
\begin{lem}\label{lemma1}
For any $v\in C^{1,\alpha}(\overline{\RN})\cap C^2(\RN)$ and for
every $R>0$ there holds
\begin{equation}\label{deltagrad}
\begin{aligned}
\int_{\Delta_R}-\Delta vX\cdot
Dv\,dX&=\frac{1-N}{2}\int_{\Delta_R}|Dv|^2dX+\int_{b_R}v_{N+1}Dv\cdot x\,dx\\
&+\int_{S_R^+}\left[\frac{R}{2}|Dv|^2-
\frac{1}{R}|Dv\cdot X|^2 \right]\,d\sigma,
\end{aligned}
\end{equation}
\begin{equation}\label{ggrad}
\int_{\Delta_R}g(v)X\cdot
Dv\,dX=-(N+1)\int_{\Delta_R}G(v)\,dX+R\int_{S_R^+}G(v)\,dx.
\end{equation}
Here $g:\R\to\R$ is any continuous function and
$G(s)=\int_0^sg(t)\,dt$.
\end{lem}
\begin{proof}
Fix $i\in\{1,\ldots,N+1\}$. Denoting by $\nu$ the outward unit
vector to $\partial\Delta_R$, we have
\[
\nu(X)=\begin{cases} \frac{X}{|X|}&\mbox{ if }x_{N+1}>0,\\
(\underbrace{0,\ldots,0}_N,-1)&\mbox{ if }x_{N+1}=0,
\end{cases}
\]
so that
\begin{equation}\label{prodotto}
X\cdot\nu(X)=\begin{cases} |X|&\mbox{ if }x_{N+1}>0,\\
0&\mbox{ if }x_{N+1}=0.
\end{cases}
\end{equation}

For any $i,j\in\{1,\ldots,N+1\}$, by Green's formula we have
\[
\int_{\Delta_R}v_iv_{ij}x_j\,dX=\frac{1}{2}\int_{\partial \Delta_R}
(v_i)^2x_j\nu_j\,d\sigma-\frac{1}{2}\int_{\Delta_R}(v_i)^2dX.
\]
Hence, denoting by $\delta_{ij}$ the usual Kronecker symbol, we have
\[
\begin{aligned}
-\int_{\Delta_R}v_{ii}v_jx_j\,dX&= -\int_{\partial
\Delta_R}v_iv_jx_j\nu_i\,d\sigma+\int_{\Delta_R}
[\delta_{ij}v_iv_j+v_iv_{ij}x_j]\,dX\\
&=\int_{\partial \Delta_R}
\left[\frac{1}{2}(v_i)^2x_j\nu_j-v_iv_jx_j\nu_i
\right]\,d\sigma\\
& +\int_{\Delta_R}\left[\delta_{ij}v_iv_j-
\frac{1}{2}(v_i)^2\right]\,dX.
\end{aligned}
\]
Summing up over $i,j\in\{1,\ldots,N+1\}$, by \eqref{prodotto}, we
get \eqref{deltagrad}.

In order to prove \eqref{ggrad}, observe that for any
$i\in\{1,\ldots,N+1\}$ we have
\[
\int_{\Delta_R}g(v)v_ix_i\,dX=\int_{\partial
\Delta_R}G(v)x_i\nu_i\,d\sigma -\int_{\Delta_R}G(v)\,dX.
\]
Summing up, using \eqref{prodotto}, we obtain \eqref{ggrad}.
\end{proof}

We now focus on the 3--dimensional case, i.e. on problem \eqref{modello}.
\begin{lem}\label{lem2}
If $v\in L^3(\R^3)$ and $\phi(x)=\frac{1}{|x|}\ast v^2$, then
\begin{equation}\label{phigrad}
\begin{aligned}
\int_{b_R} &v \phi Dv\cdot x\,dx=-\frac{3}{2}\int_{b_R}
v^2\phi\,dx+\frac{1}{16\pi}\int_{b_R}|D\phi|^2dx\\
&+\int_{\partial b_R} \left[\frac{R}{2}v^2\phi
-\frac{1}{8\pi}\left(\frac{R}{2}|D\phi|^2-\frac{1}{R}|x\cdot
D\phi|^2 \right)\right]\,d\tau.
\end{aligned}
\end{equation}
\end{lem}
\begin{proof}
Let us note that $\phi(x)\in H^2_{\rm loc}(\R^3)\cap D^1(\R^3)$ is a
solution of
\begin{equation}\label{eqphi}
-\Delta \phi=4\pi v^2 \mbox{ in }\R^3,
\end{equation}
where $D^1(\R^3)=\overline{C^\infty_C(\R^3)}^{\|\cdot\|}$, with
$\|\phi\|^2=\int_{\R^3}|D\phi|^2$ (see also \cite{teadim}).

Operating as we did to prove Lemma \ref{lemma1}, replacing
$\Delta_R$ with $b_R$, we can prove that
\begin{equation}\label{prima}
\int_{b_R}v\phi x\cdot Dv\,dx=-\frac{1}{2}\int_{b_R} v^2 x\cdot D
\phi\, dx -\frac{3}{2}\int_{b_R}  v^2 \phi\,dx
+\frac{R}{2}\int_{\partial b_R}v^2 \phi\,d\tau
\end{equation}
and
\begin{equation}\label{seconda}
-\int_{b_R}\Delta \phi x\cdot
D\phi\,dx=-\frac{1}{2}\int_{b_R}|D\phi|^2dx+\int_{\partial
b_R}\left[\frac{R}{2}|D\phi|^2-\frac{1}{R}|D\phi\cdot
x|^2\right]\,d\tau,
\end{equation}
see also \cite[Lemma 3.1]{tdnonex}.

On the other hand, from \eqref{eqphi} we have
\begin{equation}\label{ultimaeq}
4\pi \int_{b_R}v^2x\cdot D\phi\,dx=-\int_{b_R}\Delta\phi x\cdot
D\phi\,dx.
\end{equation}
Starting from \eqref{prima}, using \eqref{ultimaeq} and
\eqref{seconda}, the claim follows.
\end{proof}

\begin{lem}\label{lem3}
If $v\in H^1(\R^3)$, there exists a sequence $R_n\to \infty$ such
that
\[
\int_{\partial b_{R_n}}
\left[\frac{R_n}{2}v^2\phi
-\frac{1}{8\pi}\left(\frac{R_n}{2}|D\phi|^2-\frac{1}{R_n}|x\cdot
D\phi|^2 \right)\right]\,d\tau\to 0
\]
as $n\to \infty$.
\end{lem}
\begin{proof}
We follow the lines of \cite{BL}. First, let us note that on
$\partial b_{R_n}$ we have $\frac{1}{R_n}|x\cdot D\phi|^2\leq
\frac{1}{R_n}|x|^2|D\phi|^2=R_n|D\phi|^2\in L^1(\R^3)$. Moreover,
$v^2\in L^{6/5}(\R^3)$ by interpolation, and $\phi\in L^6(\R^3)$ by
the Sobolev inequality: hence $v^2\phi\in L^1(\R^3)$. Thus it is
enough to prove that, if $f\in L^1(\R^3)$, then there exists a
sequence $R_n\to \infty$ such that
\[
R_n\int_{\partial b_{R_n}}|f|\,d\tau\to 0.
\]
Assume this is not the case, so that there exists $\ve, R_0>0$ such
that
\[
R\int_{\partial b_R}|f|\,d\tau\geq \ve \mbox{ for every }R\geq R_0.
\]
Then
\[
\infty>\int_{\R^3}|f|\,dx=\int_0^\infty dR\int_{\partial
b_R}|f|\,d\tau\geq \int_{R_0}^\infty \frac{\ve}{R}\,dR=\infty,
\]
and a contradiction arises.
\end{proof}

\begin{lem}\label{lem4}
If $v\in H^1(\R^3)$, then
\[
\int_{\R^3} v \left(\frac{1}{|x|}\ast v^2\right) Dv\cdot
x\,dx=-\frac{5}{4}\int_{\R^3} v^2\left(\frac{1}{|x|}\ast
v^2\right)\,dx.
\]
\end{lem}
\begin{proof}
From Lemma \ref{lem2}, applying Lemma \ref{lem3}, we find
\[
\int_{\R^3} v \phi Dv\cdot x\,dx=-\frac{3}{2}\int_{\R^3}
v^2\phi\,dx+\frac{1}{16\pi}\int_{\R^3}|D\phi|^2dx.
\]
On the other hand, $\phi$ being a solution in $D^1(\R^3)$ of
\eqref{eqphi}, we get
\begin{equation}\label{identitaphi}
\int_{\R^3}|D\phi|^2dx=4\pi \int_{\R^3}v^2\phi\,dx.
\end{equation}
Substituting, we get the claim.
\end{proof}

\medskip

We are now ready to prove the variational identity \eqref{nonso} for
the generalized half Laplacian, which we believe to be quite useful
in studying system \eqref{modello}:
\begin{proof}[Proof of Theorem $\ref{identita}$]
Multiply the first equation in \eqref{modello} by $X\cdot Dv$ and
integrate on $\Delta_R$. Applying Lemma \ref{lemma1} with $g(s)=s$
we get
\begin{equation}\label{dainizio}
\begin{aligned}
0&=-\int_{\Delta_R}|Dv|^2dX+\int_{S_R^+}\left[\frac{R}{2}|Dv|^2-
\frac{1}{R}|Dv\cdot X|^2 \right]\,d\sigma \\
&+\int_{b_R}v_4Dv\cdot
x\,dx -2m^2\int_{\Delta_R}v^2dX+\frac{m^2R}{2}
\int_{S_R^+}v^2dx.
\end{aligned}
\end{equation}
By the boundary condition in \eqref{PN} we have
\[
\int_{b_R}v_4Dv\cdot x\,dx=-\int_{b_R}\left[ \omega
v+\lambda\left(\frac{1}{|x|}\ast v^2\right)v -F'(v)\right]Dv\cdot
x\,dx.
\]
Operating as in the proof of Lemma \ref{lem2} (see also \cite[Lemma
3.1]{tdnonex}), setting $\phi=\left(\frac{1}{|x|}\ast v^2\right)$,
we can prove that
\begin{equation}\label{p12}
\begin{aligned}
&\int_{b_R}\left[\omega v+\lambda\phi v -F'(v)\right]Dv\cdot x\,dx\\
&=\frac{\lambda}{16\pi}\int_{b_R}|D\phi|^2dx-\frac{3}{2}\int_{b_R}(\omega+\lambda\phi)v^2\,dx
+3\int_{b_R}F(v)\,dx\\
&+\int_{\partial b_R}\left[\frac{R}{2}(\omega+\lambda\phi)v^2\phi
-\frac{\lambda}{8\pi}\left(\frac{R}{2}|D\phi|^2-\frac{1}{R}|x\cdot
D\phi|^2\right) -RF(v)\right]\,d\tau.
\end{aligned}
\end{equation}

Substituting \eqref{p12} into \eqref{dainizio}, we obtain
\begin{equation}\label{p13}
\begin{aligned}
0&=-\int_{\Delta_R}|Dv|^2dX+\int_{S_R^+}\left[\frac{R}{2}|Dv|^2-
\frac{1}{R}|Dv\cdot X|^2 \right]\,d\sigma\\
&-\frac{\lambda}{16\pi}\int_{b_R}|D\phi|^2dx+\frac{3}{2}\int_{b_R}(\omega+\lambda\phi)v^2\,dx
-3\int_{b_R}F(v)\,dx\\
&-\int_{\partial b_R}\left[\frac{R}{2}(\omega+\lambda\phi)v^2\phi
-\frac{\lambda}{8\pi}\left(\frac{R}{2}|D\phi|^2-\frac{1}{R}|x\cdot
D\phi|^2\right) -RF(v)\right]\,d\tau\\
&-2m^2\int_{\Delta_R}v^2dX+\frac{m^2R}{2}
\int_{S_R^+}v^2dx.
\end{aligned}
\end{equation}

As in the proof of Lemma \ref{lem3} we can find a sequence $R_n\to
\infty$ such that the integrals over $\partial b_{R_n}$ and over
$S_R^+$ go to 0 as $n\to \infty$. In this way \eqref{p13} gives
\[
\begin{aligned}
0&=-\int_{\R^4_+}|Dv|^2dX-2m^2\int_{\R^4_+}v^2dX\\
&
-\frac{\lambda}{16\pi}\int_{\R^3}|D\phi|^2dx+\frac{3}{2}\int_{\R^3}(\omega+\lambda\phi)v^2\,dx
-3\int_{\R^3}F(v)\,dx.
\end{aligned}
\]

By substituting in the equation above the term $\int |D\phi|^2$
taken from \eqref{identitaphi}, we finally get \eqref{nonso}.
\end{proof}

We are now ready for the
\begin{proof}[Proof of Theorem $\ref{nonex}$]
Since $v$ is a solution of \eqref{PN}, we have
\begin{equation}\label{diparte}
\int_{\R^4_+}(|Dv|^2+m^2v^2)\,dX=\int_{\R^3}[\omega
v+\lambda\phi v-F'(v)]v\,dx.
\end{equation}

Now we isolate $\int_{\R^4_+}|Dv|^2\,dX$ in \eqref{diparte} and
substituting in \eqref{nonso}, we get
\[
\begin{aligned}
0&=-m^2\int_{\R^4_+}v^2dX+\frac{\omega}{2}
\int_{\R^3}v^2dx+\frac{1}{4}\lambda
\int_{\R^3}v^2\phi\,dx\\
&-\int_{\R^3}\big[3F(v)-F'(v)v\big]\,dx.
\end{aligned}
\]
Thus, if \eqref{ipo1W} holds, we get $v\equiv 0$.

On the other hand, if we isolate $\int_{\R^4_+}v^2\,dX$ in
\eqref{diparte} and we substitute in \eqref{nonso}, we get
\[
\begin{aligned}
0&=\int_{\R^4_+}|Dv|^2dX-\frac{\omega}{2}\int_{\R^3}v^2dx-\frac{3}{4}\lambda
\int_{\R^3}v^2\phi\,dx\\
&+\int_{\R^3}\big[2F'(v)v-3F(v)\big]\,dx.
\end{aligned}
\]
Thus, if \eqref{ipo2W} holds, again we get $v\equiv 0$.

Now, if $\omega>0$ and $\lambda\leq 0$, starting from \eqref{nonso}, using \eqref{C2}, we have
\[
\begin{aligned}
0&\leq \left(-1+\frac{3\omega}{2\mu}\right)\int_{\R^4_+}|Dv|^2dX+\left(\frac{3\omega \mu}{2}-2m^2\right)\int_{\R^4_+}v^2dX
\\
&+\frac{5}{4}\lambda\int_{\R^3}v^2\phi\,dx-3\int_{\R^3}F(v)\,dx.
\end{aligned}
\]
Choosing $\mu\in
\left[\frac{3\omega}{2},\frac{4m^2}{3\omega}\right]$, by
\eqref{ipo3W} we obtain again $v\equiv 0$. If $\omega=0$ the
conclusion is easier.

The last statement needs a longer proof and the following estimate,
which also establishes a non vanishing property for nontrivial
solutions of \eqref{modello}, and for whose proof we direct the
reader to the Appendix:
\begin{lem}\label{stimala}
If $v\in H^1(\RN)$ solves \eqref{PN} with $\lambda>0$, $\omega\in
(0,m)$ and $F_s(x,s)s\geq 0$ for every $s\in \R$ and a.e. $x\in
\R^N$, then
\begin{equation}\label{stimata}
\|v\|^2\leq \frac{m}{m-\omega}\lambda \int_{\R^N}\big(W\ast v^2\big)
v^2\,dx
\end{equation}
and there exists $C=C(m,\omega)>0$ such that
\begin{equation}\label{stimabasso}
\|v\|^2\geq \frac{C}{\lambda}
\end{equation}
for any nontrivial solution of \eqref{modello}.
\end{lem}

Let us start noting that for any solution $v$ and for any $\rho>0$
we have
\begin{equation}\label{ultimissima}
\begin{aligned}
0&=\eqref{nonso}+\rho J'(v)v\\
& =(\rho-1)\int_{\R^4_+}|Dv|^2dX+(\rho-2)m^2\int_{\R^4_+}v^2dX+
\left(\frac{3}{2}-\rho\right)\omega \int_{\R^3}v^2dx\\
&+\left(\frac{5}{4}-\rho\right)\lambda
\int_{\R^3}v^2\phi\,dx+\int_{\R^3}[\rho F'(v)v-3F(v)]\,dx.
\end{aligned}
\end{equation}

Now, take any $\rho<3/2$ and start from \eqref{ultimissima};
applying \eqref{C2} with $\ve>0$, we obtain
\begin{equation}\label{mancava}
\begin{aligned}
0&\leq \left[\rho-1+\ve \omega\left(\frac{3}{2}-\rho\right)\right]
\int_{\RN}|Dv|^2dxdx_{N+1}\\
&+m^2 \left[ \rho-2+\frac{\omega}{\ve
m^2}\left(\frac{3}{2}-\rho\right)\right]\int_{\RN}v^2dxdx_{N+1}\\
&+\left(\frac{5}{4}-\rho\right)\lambda
\int_{\R^N}v^2\phi\,dx+\int_{\R^3}[\rho F'(v)v-3F(v)]\,dx.
\end{aligned}
\end{equation}
By choosing $\ve$ small enough, we can suppose that
\[
\rho-1+\ve \omega\left(\frac{3}{2}-\rho\right)\leq
\rho-2+\frac{\omega}{\ve m^2}\left(\frac{3}{2}-\rho\right),
\]
so that \eqref{mancava} becomes
\[
0\leq \left[ \rho-2+\frac{\omega}{\ve
m^2}\left(\frac{3}{2}-\rho\right)\right]\|v\|^2+
\left(\frac{5}{4}-\rho\right)\lambda
\int_{\R^N}v^2\phi\,dx+\int_{\R^3}[\rho F'(v)v-3F(v)]\,dx.
\]
By applying \eqref{stimata} with $W(x)=1/|x|$ and $N=3$, we find
\begin{equation}\label{bellina}
\begin{aligned}
0&\leq \left\{\left[ \rho-2+\frac{\omega}{\ve
m^2}\left(\frac{3}{2}-\rho\right)\right]\frac{m}{m-\omega}+
\left(\frac{5}{4}-\rho\right)\right\}\lambda
\int_{\R^N}v^2\phi\,dx\\
&+\int_{\R^3}[\rho F'(v)v-3F(v)]\,dx.
\end{aligned}
\end{equation}
After some calculations, we find that the coefficient of $\int
v^2\phi$ in the inequality above is non positive iff
\[
\rho \frac{\omega}{m}\left(\frac{1}{\ve m}-1\right)\geq
\frac{3\omega}{2\ve m^2} -\frac{5\omega}{4m}-\frac{3}{4}.
\]
If $\ve<1/m$ is small enough, it is clear that both sides of the
previous inequality are positive, so that we are allowed to choose
\[
\rho\geq \left(6-5\ve m-\frac{3\ve m^2}{\omega}\right)
\frac{1}{4(1-\ve m)}.
\]
Let us remark that the function $g(\ve)=(6-5\ve m-3\ve
m^2/\omega)[4(1-\ve m)]^{-1}$ is strictly decreasing in $(0,1/m)$,
and that $g(0^+)=3/2$, so that a possible choice on $\rho<3/2$ is
given for $\ve<1/m$.

Therefore, passing to the limit as $\rho\uparrow 3/2$ in
\eqref{bellina}, we get
\[
0\leq \frac{\omega -3m}{4m-4\omega}\lambda \int_{\R^N}v^2\phi\,dx+
\frac{3}{2}\int_{\R^3}[F'(v)v-2F(v)]\,dx \leq  \frac{\omega
-3m}{4m-4\omega}\lambda \int_{\R^N}v^2\phi\,dx
\]
by assumption on $F$ appearing in \eqref{ipo4W}. Being the remaining
coefficient a strictly negative number, we get $\int v^2\phi=0$, and
from \eqref{stimata} also $v\equiv0$.

Theorem \ref{nonex} is now completely proved.
\end{proof}

We conclude with the
\begin{proof}[Proof of Theorem $\ref{generale}$]
The first part is very similar to the proof of Theorem \ref{nonex},
which was obtained adding \eqref{nonso}$+\rho J'(v)v$ for $\rho=1$
and $\rho=2$. Now, starting from \eqref{ultimissima}, if
\eqref{ipo1Wrho} holds, all the coefficients are less or equal to 0,
and we get $v\equiv0$; if \eqref{ipo2Wrho} holds, all the
coefficients in \eqref{ultimissima} are nonnegative, and we obtain
again $v\equiv 0$.

Now assume the \eqref{ipo3Wrho} holds. Since $\omega\in
\left(0,2m\frac{\sqrt{(\rho-1)(\rho-2)}}{2\rho-3}\right]$ and
$\lambda\leq0$, starting from \eqref{ultimissima} and using
\eqref{C2} and the fact that $\rho>2$, we find
\[
0\geq
\left[\rho-1+\left(\frac{3}{2}-\rho\right)\frac{\omega}{\ve}\right]
\int_{\R^4_+}|Dv|^2dX+\left[(\rho-2)m^2+\left(\frac{3}{2}-
\rho\right)\ve\omega\right] \int_{\R^4_+}v^2dX.
\]
Both the coefficients of the integrals above are nonnegative
provided that
\[
\frac{2\rho-3}{\rho-1}\frac{\omega}{2}\leq \ve\le
2\frac{(\rho-2)m^2}{(2\rho-3)\omega},
\]
which is possible by the bound on $\omega$.
\end{proof}

\appendix
\section{Appendix}

\begin{proof}[Proof of Lemma $\ref{genere}$]
Fix $k\in \N$ and take any subspace $H_k$ of $H^1_r(\RN)$ having
dimension $k$. We first prove that $M_k:=M\cap H_k$, which is non
empty by the definition of the map in \eqref{psimia}, is bounded.
Indeed, assume by contradiction that there exists an unbounded
sequence $(v_n)_n$ in $M_k$. Without loss of generality, we can
assume that $\frac{v_n}{\|v_n\|}\to v\in H_k$, where $\|v\|=1$.

Setting $U_n=v_nW\ast v_n^2$, then
\[
\int_{\R^N}U_nv_ndx=\int_{\R^N}v_n^2W\ast v_n^2dx=1.
\]
Dividing both sides by $\|v_n\|^4$, we obtain
\[
0=\lim_{n\to \infty}\int_{\R^N}\frac{U_nv_ndx}{\|v_n\|^4}=\lim_{n\to
\infty} \int_{\R^N}
 \frac{v_n^2}{\|v_n\|^2} W\ast  \frac{v_n^2}{\|v_n\|^2}dx=\int_{\R^N}v^2W\ast
 v^2dx>0
\]
since $v$ is nontrivial, and we get a contradiction.

$M_k$ being bounded and closed in $H_k$, we get that $M_k$ is a
compact subset of $H_k$.

We now show that $\gamma(M_k)\geq k$, and being $k$ arbitrary, we
deduce that $\tilde{\gamma}(\mathcal M_\sigma)=\infty$, and the
lemma is proved. For this, let us consider the map $\pi:S^{k-1}\to
M$, where $S^{k-1}$ denotes the unit sphere of $H_k$, defined as
$\pi(u)=u_M$, $u_M$ being the unique point of the half line $\R^+u$
intersecting $M$; note that the definition of $\pi$ is well posed by
\eqref{psimia}. Clearly $\pi$ is an odd continuous map with
$\pi(S^{k-1})=M_k$; thus, by \cite[Proposition 5.4
$(4^\circ)$]{Struwe}, we get $\gamma(M_k)\geq \gamma(S^{k-1})$,
while $\gamma(S^{k-1})=k$ by \cite[Proposition 5.2]{Struwe}.
\end{proof}

\begin{proof}[Proof of Lemma $\ref{stimala}$]
In \eqref{defsol} take $w=v$, obtaining
\[
\begin{aligned}
\|v\|^2&=\omega \int_{\R^N}v^2dx+\lambda \int_{\R^N}\big(W\ast
v^2\big)v^2dx-\int_{\R^N}F_s(x,v)v\,dx\\
&\leq \omega \int_{\R^N}v^2dx+\lambda \int_{\R^N}\big(W\ast
v^2\big)v^2dx
\end{aligned}
\]
by the assumption on $F_s$. Using \eqref{C2} with $\ve=m$ we obtain
\[
\|v\|^2\leq \frac{\omega}{m} \int_{\RN}|Dv|^2dxdx_{N+1}+m\omega
\int_{\RN}v^2dxdx_{N+1}+\lambda \int_{\R^N}\big(W\ast v^2\big)v^2dx,
\]
from which \eqref{stimata} follows.

From \eqref{stimata}, by applying \eqref{maggconv}, we immediately
get \eqref{stimabasso} for any solution $v\not \equiv 0$.
\end{proof}


\begin{thebibliography}{99}

\bibitem{adams}
{\sc Adams, R.A.}: {\it Sobolev spaces}. Pure and Applied
Mathematics, Vol. 65. Academic Press, New York-London, 1975.

\bibitem{ar}
{\sc Ambrosetti, A., Rabinowitz, P.H.}: Dual variational methods in
critical point theory and applications. {\it J. Funct. Analysis}
{\bf14}, 349--381 (1973).

\bibitem{amru}
{\sc Ambrosetti, A., Ruiz, D.}: Multiple bound states for the
Schr\"odinger--Poisson problem. {\it Commun. Contemp. Math}
{\bf 10} 391--404 (2008).

\bibitem{bftop}
{\sc Benci, V., Fortunato, D.}: An eigenvalue problem for the
Schr\"odinger--Maxwell equations. {\it Top. Meth. Nonlinear Anal.}
{\bf11}, 283--293 (1998).

\bibitem{bf}
{\sc Benci, V., Fortunato, D.}: Solitary waves in abelian gauge
theories. {\it Adv. Nonlinear Stud.} {\bf 8}, 327--352 (2008).

\bibitem{bfcmp}
{\sc Benci, V., Fortunato, D.}: Spinning $Q$--balls for the
Klein-Gordon-Maxwell equations. {\it Comm. Math. Phys.} {\bf 295}
639--668 (2010).

\bibitem{bfarc}
{\sc Benci, V., Fortunato, D.}: Towards a Unified Field Theory for
Classical Electrodynamics. {\it Arch. Rational Mech. Anal.}
{\bf 173}, 379--414 (2004).

\bibitem{BL}
{\sc Berestycki, H., Lions, P.L.}: Nonlinear scalar field equatios,
I - Existence of a ground state. {\it Arch. Rat. Mech. Anal.} {\bf
82}, 313--345 (1983).

\bibitem{CT}
{\sc Cabr\'e, X., Tan, J.}: Positive solutions of nonlinear problems
involving the square root of the Laplacian. {\it Adv. Math.} {\bf 224},
2052--2093 (2010).

\bibitem{cdl}
{\sc Canc\`es, \'E., Deleurence, A., Lewin, M.}: A New Approach to
the Modeling of Local Defects in Crystals: The Reduced Hartree-Fock
Case. {\it Commun. Math. Phys.} {\bf281}, 129--177 (2008).

\bibitem{ce}
{\sc Canc\`es, \'E., Ehrlacher, V.}: Local Defects are Always
Neutral in the Thomas-Fermi-von Weisz\"acker Theory of Crystals.
{\it Arch. Rational Mech. Anal.} {\bf 202}, 933--973 (2011).

\bibitem{cale}
{\sc Canc\`es, \'E., Lewin, M.}:  The Dielectric Permittivity of
Crystals in the Reduced Hartree--Fock Approximation. {\it Arch.
Rational Mech. Anal.} {\bf 197}, 139--177 (2010).

\bibitem{calbli}
{\sc Catto, I., Le Bris, C., Lions, P.-L.}: On some periodic
Hartree--type models for crystals. {\it Ann. Inst. H. Poincar\'e
Anal. Non Lin\'eaire} {\bf19}, 143--190 (2002).

\bibitem{czn}
{\sc Coti Zelati, V., Nolasco, M.}: Existence of ground states for
nonlinear, pseudorelativistic Schr\"odinger equations. {\it Rend.
Lincei Mat. Appl.} {\bf22}, 51--72 (2011).

\bibitem{tdnonex}
{\sc D'Aprile, T., Mugnai, D.}: Non--existence results for the
coupled Klein--Gordon--Maxwell equations. {\it Adv. Nonlinear Stud.}
{\bf 4}, 307--322 (2004).

\bibitem{teadim}
{\sc D'Aprile, T., Mugnai, D.}: Solitary Waves for nonlinear
Klein--Gordon--Maxwell and Schr\"odinger--Maxwell equations. {\it
Proc. R. Soc. Edinb. Sect. A} {\bf134}, 1--14 (2004).

\bibitem{2}
{\sc Elgart, A., Schlein, B.}: Mean field dynamics of boson stars.
{\it Comm. Pure Appl. Math.} {\bf 60}, 500--545 (2007).

\bibitem{FJL}
{\sc Fr\"ohlich, J., Jonsson, B.L.G., Lenzmann, E.}: Boson stars as
solitary waves. {\it Comm. Math. Phys.} {\bf 274}, 1--30 (2007).

\bibitem{FL}
{\sc Fr\"ohlich, J., Lenzmann, E.}: Blowup for Nonlinear Wave Equations
Describing Boson Stars. {\it Comm. Pure Appl. Math.} {\bf 60}, 1691--1705 (2007).

\bibitem{LL}
{\sc Lieb, E.H., Loss, M.}: {\it Analysis}. Graduate Studies in
Mathematics 14, American Mathematical Society, 1997.

\bibitem{11}
{\sc Lieb, E.H., Yau, H.-T.}: The Chandrasekhar theory of stellar
collapse as the limit of quantum mechanics. {\it Comm. Math. Phys.} {\bf 112},
147--174 (1987).

\bibitem{MaZhao}
{\sc Ma, L., Zhao, L.}: Classification of positive solitary
solutions of the nonlinear Choquard equation. {\it Arch. Ration.
Mech. Anal.} {\bf 195},455--467 (2010).

\bibitem{mug}
{\sc Mugnai, D.}: Coupled Klein--Gordon and Born--Infeld type
equations: looking for solitons. {\it R. Soc. Lond. Proc. Ser. A}
{\bf460}, 1519--1528 (2004).

\bibitem{nuovo}
{\sc Mugnai, D.}: Solitary waves in Abelian Gauge Theories with
strongly nonlinear potentials. {\it Ann. Inst. H. Poincar\'e Anal.
Non Lin\'eaire} {\bf 27}, 1055--1071 (2010).

\bibitem{SPpos}
{\sc Mugnai, D.}: The Schr\"odinger--Poisson system with positive
potential. {\it Comm. Partial Differential Equations} {\bf 36},
1099--1117 (2011).

\bibitem{PA}
{\sc Palais, R.S.}: The principle of symmetric criticality. {\it
Comm. Math. Phys.} {\sc 69}, 19--30 (1979).

\bibitem{Rab}
{\sc Rabinowitz, P.H.}: {\it Mini-max Methods in Critical Point
Theory with Applications to Differential Equations}. CBMS Reg. Conf.
Ser. in Math. No. 65, AMS, Providence R.I., 1986.

\bibitem{ru}
{\sc Ruiz, D.}: The Schr\"odinger--Poisson equation under the
effect of a nonlinear local term. {\it J. Funct. Anal.}
{\bf 237}, 655--674 (2006).

\bibitem{ruarrso}
{\sc Ruiz Arriola, E., Soler, J.}: A variational approach to the
Schr\"odinger--Poisson system: asymptotic behaviour, breathers and
stability. {\it J. Stat. Phys.} {\bf 103}, 1069--1106 (2001).

\bibitem{ss}
{\sc Sickel W., Skrzypczak, L.}: Radial subspaces of Besov and
Lizorkin-Triebel spaces: extended Strauss lemma and compactness of
embeddings. {\it J. Fourier Anal. Appl.} {\bf6}, 639--662 (2000).

\bibitem{Struwe}
{\sc Struwe, M.}: {\it Variational methods. Applications to
nonlinear partial differential equations and Hamiltonian systems}.
Springer-Verlag, Berlin, 4th edition, 2008.

\bibitem{willem}
{\sc Willem, M.}: {\it Minimax Theorems}. Progr. Nonlinear
Differential Equations Appl. 24. Birkhäuser Boston, Inc., Boston,
MA, 1996.

\end{thebibliography}
\end{document}